\documentclass[11pt]{article}
\usepackage[arrow, matrix]{xy}
\usepackage{amsmath}
\usepackage{cancel}
\usepackage{amssymb}
\usepackage{amsthm}
\usepackage[mathscr]{eucal}
\input{epsf}
\theoremstyle{definition}
\newtheorem{definition}{Definition}

\newtheorem{lemma}[definition]{Lemma}

\theoremstyle{remark}

\newcounter{enumctr}

\newcommand{\R}{\mathbb{R}}
\newcommand{\N}{\mathbb{N}}
\newcommand{\C}{\mathbb{C}}

\renewcommand{\phi}{\varphi}

\begin{document}
\title{\vspace*{-10mm}
On Some Special Properties\\
of Mittag-Leffler Functions}
\author{
H.T.~Tuan\footnote{\tt httuan@math.ac.vn, \rm Institute of Mathematics, Vietnam Academy of Science and Technology, 18 Hoang Quoc Viet, 10307 Ha Noi, Viet Nam}}
\date{24 January 2016}
\maketitle
\begin{abstract}
Our aim in this report is to investigate the asymptotic behavior of Mittag-Leffle functions. We give some estimates involving the Mittag-Leffler functions and their derivatives.
\end{abstract}
\section{Mittag--Leffler functions}\label{MLF}
Similar to the exponential functions frequently used in investigation of integer-order systems, Mittag-Leffle functions naturally occur in the solutions of fractional differential equations. In this report, we establish some estimates involving the Mittag-Leffler functions and their derivatives. The results are needed for the proofs of the stability theorems presented in \cite{Cong}, \cite{Cong_1}, \cite{Cong_2}, and \cite{Cong_3}. As a preparation, following Podlubny~\cite{Podlubny} we introduce some notations as follows.

Let $\gamma(\varepsilon,\theta)$, $\varepsilon>0,\,0<\theta\le \pi$, denote the contour in the complex plane which consists of the three parts
\begin{itemize}
\item [(i)] $\text{arg}(z)=-\theta$, $|z|\ge \varepsilon$;
\item [(ii)] $-\theta\le \text{arg}(z)\le \theta$, $|z|=\varepsilon$;
\item [(iii)] $\text{arg}(z)=\theta$, $|z|\ge \varepsilon$.
\end{itemize}
This contour $\gamma(\varepsilon,\theta)$ divides the complex plane $(z)$ into two domains, which we denote by $G^{-}(\varepsilon,\theta)$ and $G^{+}(\varepsilon,\theta)$. These domains lie correspondingly on the left and on the right side of the contour $\gamma(\varepsilon,\theta)$.
\begin{lemma}\label{lemma1}
Let $0<\alpha<1$ and $\beta$ be an arbitrary complex number. Then for an arbitrary $\varepsilon>0$ and $\frac{\pi\alpha}{2}<\theta<\pi\alpha$, we have
\begin{itemize}
\item [(i)] $$E_{\alpha,\beta}(z)=\frac{1}{2\alpha\pi i}\int_{\gamma(\varepsilon,\theta)}\frac{\exp{(\zeta^{\frac{1}{\alpha}})}\zeta^{\frac{1-\beta}{\alpha}}}{\zeta-z}d\zeta,\,\, z\in G^{-}(\varepsilon,\theta);$$
\item [(ii)] \begin{align*}
E_{\alpha,\beta}(z)=&\dfrac{1}{\alpha}z^{\frac{1-\beta}{\alpha}}\exp{(z^{\frac{1}{\alpha}})}+\frac{1}{2\alpha\pi i}\int_{\gamma(\varepsilon,\theta)}\frac{\exp{(\zeta^{\frac{1}{\alpha}})}\zeta^{\frac{1-\beta}{\alpha}}}{\zeta-z}d\zeta,\\
& z\in G^{+}(\varepsilon,\theta).
\end{align*}
\end{itemize}
\end{lemma}
\begin{proof}
See Podlubny~\cite[Theorem 1.1, p.\ 30]{Podlubny}
\end{proof}
\begin{lemma}\label{lemma3}
Let $\lambda\in\mathbb C\setminus\{0\}$, $\frac{\alpha \pi}{2}<\theta< \alpha\pi$ and $\theta_0 \in (0,\theta-\frac{\alpha\pi}{2})$ satisfying $$| \theta-|\text{arg}(\lambda)||\ge\theta_0.$$ We define
\begin{align*}
m(\alpha,\lambda)&=\max\left\{\frac{\int_{\gamma(1,\theta)}|\exp{(\zeta^{\frac{1}{\alpha}})}|d\zeta}{2\alpha\pi |\lambda|\sin\theta_0},\frac{\int_{\gamma(1,\theta)}|\exp{(\zeta^{\frac{1}{\alpha}})}\zeta^{\frac{1}{\alpha}}|d\zeta}{2\alpha|\lambda|^2\sin\theta_0}\right\},\\
t_0&=\dfrac{1}{|\lambda|^{\frac{1}{\alpha}}(1-\sin\theta_0)^{\frac{1}{\alpha}}}.
\end{align*}
Then 
\begin{itemize}
\item [(i)] If $\lambda\in\{z\in\mathbb{C}:|\text{arg}(z)|<\frac{\alpha\pi}{2}\}$ and $t\geq t_0,$ we have
$$\left|E_\alpha(\lambda t^{\frac{1}{\alpha}})-\frac{1}{\alpha}\exp{(\lambda^{\frac{1}{\alpha}}t)}\right|\le \frac{m(\alpha,\lambda)}{t^{\alpha}};$$
\item [(ii)] If $\lambda\in\{z\in\mathbb{C}:|\text{arg}(z)|<\frac{\pi \alpha}{2}\}$ and $t\geq t_0,$ we have
$$\left|t^{\alpha-1}E_{\alpha,\alpha}(\lambda t^{\alpha})-\frac{1}{\alpha}\lambda^{\frac{1}{\alpha}-1}
\exp{(\lambda^{\frac{1}{\alpha}}t)}\right|\le \frac{m(\alpha,\lambda)}{t^{\alpha+1}};$$
\item [(iii)] If $\lambda\in\{z\in\mathbb{C}:\theta<|\text{arg}(z)|\le\pi\}$ and $t\geq t_0,$ we have
$$|t^{\alpha-1}E_{\alpha,\alpha}(\lambda t^{\alpha})|\le \frac{m(\alpha,\lambda)}{t^{\alpha+1}}.$$
\end{itemize}
\end{lemma}
\begin{proof}
\noindent (i) Let $\varepsilon=1$, from Lemma \ref{lemma1}(ii) we have
$$E_{\alpha}(z)=\frac{1}{\alpha}\exp{(z^{\frac{1}{\alpha}})}+\frac{1}{2\alpha\pi i}\int_{\gamma(1,\theta)}\frac{\exp{(\zeta^{\frac{1}{\alpha}})}}{\zeta-z}d\zeta,\,\, z\in G^{+}(1,\theta).$$
On the other hand, for $z\in G^{+}(1,\theta-\theta_0)$ and $\zeta\in \gamma(1,\theta)$, we have
$$|\zeta-z|\ge\min\{|z|\sin\theta_0,|z|-1\}.$$
Hence, for $z\in G^{+}(1,\theta-\theta_0)$ and $|z|\ge\frac{1}{1-\sin\theta_0}$, we obtain the estimate
$$\left|E_\alpha(z)-\frac{1}{\alpha}\exp{(z^{\frac{1}{\alpha}})}\right|\le \frac{\int_{\gamma(1,\theta)}|\exp{(\zeta^{\frac{1}{\alpha}})}|d\zeta}{2\pi\alpha\sin\theta_0}\frac{1}{|z|}.$$
Substituting $z=\lambda t^{\alpha}$, where $\lambda\in\{z\in\C \setminus\{0\}:|\text{arg(z)}| \le \theta-\theta_0\}$ and $t\ge t_0$, we have
$$|E_\alpha(\lambda t^{\alpha})-\frac{1}{\alpha}\exp{(\lambda^{\frac{1}{\alpha}}t)}|\le \frac{\int_{\gamma(1,\theta)}|\exp{(\zeta^{\frac{1}{\alpha}})}|d\zeta}{2\alpha\pi |\lambda|\sin\theta_0}\frac{1}{t^{\alpha}}.$$
\noindent (ii) By virtue of Lemma \ref{lemma1}(ii), we have
$$E_{\alpha,\alpha}(z)=\frac{1}{\alpha}z^{\frac{1-\alpha}{\alpha}}\exp{(z^{\frac{1}{\alpha}})}+\frac{1}{2\alpha\pi i}\int_{\gamma(1,\theta)}\frac{\exp{(\zeta^{\frac{1}{\alpha}})}\zeta^{\frac{1-\alpha}{\alpha}}}{\zeta-z}d\zeta,\,\, z\in G^{+}(1,\theta).$$
Using the identity $\frac{1}{\zeta-z}=-\dfrac{1}{z}+\frac{\zeta}{z(\zeta-z)}$ yields
$$E_{\alpha,\alpha}(z)=\frac{1}{\alpha}z^{\frac{1-\alpha}{\alpha}}\exp{(z^{\frac{1}{\alpha}})}+\frac{1}{2\alpha\pi i}\int_{\gamma(1,\theta)}\frac{\exp{(\zeta^{1/\alpha})}\zeta^{\frac{1}{\alpha}}}{\zeta-z}d\zeta,\,\, z\in G^{+}(1,\theta).$$
This implies
\begin{equation}\label{Eq_exam}
\left|E_{\alpha,\alpha}(z)-\frac{1}{\alpha}z^{\frac{1-\alpha}{\alpha}}\exp{(z^{\frac{1}{\alpha}})}\right|\le \frac{\int_{\gamma(1,\theta)}|\exp{(\zeta^{\frac{1}{\alpha}})}\zeta^{\frac{1}{\alpha}}|d\zeta}{2\alpha \pi \sin\theta_0}\frac{1}{|z|^2},
\end{equation}
where $z\in G^{+}(1,\theta-\theta_0)$ and $|z|\ge\frac{1}{1-\sin\theta_0}.$
Let $z=\lambda t^{\alpha}$ with $\lambda\in\{z\in\C\setminus\{0\}:|\text{arg(z)}|\le \theta-\theta_0\}$ and $t\ge t_0,$ we have
$$\left|t^{\alpha-1}E_{\alpha,\alpha}(\lambda t^\alpha)-\frac{1}{\alpha}\lambda^{\frac{1-\alpha}{\alpha}}\exp{(\lambda^{\frac{1}{\alpha}}t)}\right|\le \frac{\int_{\gamma(1,\theta)}|\exp{(\zeta^{\frac{1}{\alpha}})}\zeta^{\frac{1}{\alpha}}|d\zeta}{2\pi\alpha |\lambda|^2\sin\theta_0}\frac{1}{t^{\alpha+1}}.$$
\noindent (iii) According to Lemma \ref{lemma1}(i), we have
 $$E_{\alpha,\alpha}(z)=\frac{1}{2\alpha\pi i}\int_{\gamma(1,\theta)}\frac{\exp{(\zeta^{\frac{1}{\alpha}})}\zeta^{\frac{1-\alpha}{\alpha}}}{\zeta-z}d\zeta,\,\, z\in G^{-}(1,\theta+\theta_0).$$
Using the identity $\frac{1}{\zeta-z}=-\frac{1}{z}+\frac{\zeta}{z(\zeta-z)}$ yields
$$E_{\alpha,\alpha}(z)=\frac{1}{2\alpha\pi i}\int_{\gamma(1,\theta)}\frac{\exp{(\zeta^{\frac{1}{\alpha}})}\zeta^{\frac{1}{\alpha}}}{\zeta-z}d\zeta,\,\, z\in G^{-}(1,\theta+\theta_0).$$
$$|E_{\alpha,\alpha}(z)|\le \frac{\int_{\gamma(1,\theta)}|\exp{(\zeta^{\frac{1}{\alpha}})}\zeta^{\frac{1}{\alpha}}|d\zeta}{2\alpha \pi \sin\theta_0}\frac{1}{|z|^2},$$
for $z\in G^{-}(1,\theta+\theta_0)$ and $|z|\ge\frac{1}{1-\sin\theta_0}.$
Substituting $z=\lambda t^{\alpha}$, we have
$$|t^{\alpha-1}E_{\alpha,\alpha}(\lambda t^{\alpha})|\le \frac{\int_{\gamma(1,\theta)}|\exp{(\zeta^{\frac{1}{\alpha}})}\zeta^{\frac{1}{\alpha}}|d\zeta}{2\pi\alpha |\lambda|^2\sin\theta_0}\frac{1}{t^{\alpha+1}},$$
where $\lambda\in\{z\in\C\setminus\{0\}:\theta+\theta_0\le|\text{arg(z)}|\le\pi\}$ and $t\ge t_0.$ The proof is complete.
\end{proof}
\begin{lemma}\label{LimitLemma}
For any function  $g\in C_\infty(\R_{\geq 0};\R)$ and $\lambda\in\{z\in\C\setminus\{0\}:|\text{arg(z)}|<\frac{\pi\alpha}{2}\}$, the following limit holds
\begin{align}\label{Eq4}
\lim_{u\to\infty}&\int_0^u
\notag (u-s)^{\alpha-1}\frac{E_{\alpha,\alpha}(\lambda(u-s)^\alpha)}{E_\alpha(\lambda u^\alpha)}g(s)\;ds\\
&=\lambda^{\frac{1}{\alpha}-1}\int_0^\infty\exp(-\lambda^{\frac{1}{\alpha}}s)g(s)\;ds.
\end{align}
\end{lemma}
\begin{proof}
The proof of this lemma follows imediately by using Lemma \ref{lemma3} and arguments analogous that used for the proof of Lemma 8 in \cite{Cong}. Hence, we omit it.
\end{proof}
\begin{lemma}\label{thm12}
Let $\lambda \in \C\setminus\{0\}$ with $\frac{\alpha \pi }{2}< |\text{arg}(\lambda)|\leq \pi$ and $l\in \N$. Then, there exists positive constant $M_l(\alpha,\lambda), \hat{M}_l(\alpha,\lambda)$ and a positive real number $t_0$ such that the following statements hold
\begin{itemize}
\item [(i)] $|\dfrac{d^l}{d\lambda^l} E_{\alpha}(\lambda t^{\alpha})| \le \frac{M_l(\alpha,\lambda)}{t^{\alpha}}\quad  \hbox{for any } t>t_0,$
\item [(ii)] $|\dfrac{d^l}{d\lambda^l} E_{\alpha,\alpha}(\lambda t^{\alpha})| \le \frac{\hat {M}_l(\alpha,\lambda)}{t^{2\alpha}}\quad  \hbox{for any } t>t_0.$
\end{itemize}
\end{lemma}
\begin{proof}
First, we note that since $\lambda$ satisfies by assumption the inequalities
 $$ 
\frac{\alpha \pi}{2}< |\arg{(\lambda)}| < \pi,
$$ 
we can choose and fix $\frac{\alpha \pi}{2} < \theta < \alpha\pi$ and $0<\theta_0<\frac{\pi}{2}$ such that $|\arg{(\lambda)}|>\theta +\theta_0$. This implies $\lambda t^\alpha \in G^{-}(1, \theta +\theta_0)$ for any $t>0$. Moreover, by direct computation, we see 
\begin{equation}\label{InE}
\inf_{\zeta \in \gamma(1,\theta)}|\zeta-\lambda t^\alpha|\ge |\lambda|t^\alpha \sin \theta_0,
\end{equation}
for any $t>t_0:=\dfrac{1}{|\lambda|^{\frac{1}{\alpha}}(1-\sin\theta_0)^{\frac{1}{\alpha}}}$.

\noindent (i) With $\theta$ defined as above, for any $t>0$, since $\lambda t^\alpha\in G^-(1,\theta)$ using Lemma~\ref{lemma1} we have
\[
E_{\alpha}(\lambda t^\alpha)=\dfrac{1}{2\alpha\pi i}\int_{\gamma(1,\theta)}\frac{\exp{(\zeta^{\frac{1}{\alpha}})}}{\zeta-\lambda t^\alpha}d\zeta.
\]
Hence
\begin{align*}
\frac{d^l}{d\lambda^l}E_{\alpha}(\lambda t^\alpha)&=\dfrac{1}{2\alpha\pi i}\int_{\gamma(1,\theta)}\exp{(\zeta^{\frac{1}{\alpha}})}\frac{d^l}{d\lambda^l}(\zeta-\lambda t^\alpha)^{-1}d\zeta\\
&=\frac{t^{\alpha l}l!}{2\alpha \pi i}\int_{\gamma(1,\theta)}\frac{\exp{(\zeta^{\frac{1}{\alpha}})}}{(\zeta-\lambda t^\alpha)^{l+1}}d\zeta,
\end{align*}
which together with \eqref{InE} implies
\begin{align}\label{InE1}
|\frac{d^l}{d\lambda^l}E_{\alpha}(\lambda t^\alpha)|\le \frac{l!}{2\alpha \pi (|\lambda|\sin \theta_0)^{l+1}}\times \frac{1}{t^\alpha}\int_{\gamma(1,\theta)}|\exp{(\zeta^{\frac{1}{\alpha}})}|d\zeta,
\end{align}
provided that $t >  t_0$.
Note that the integral in the right-hand side of \eqref{InE1} converges because for $\zeta$ such that $\arg(\zeta)=\pm \theta$ and $|\zeta|\geq 1$ we have
$$
\left| \exp(\zeta^\frac{1}{\alpha})\right| = \exp\left(|\zeta^\frac{1}{\alpha}|\cos\Big(\frac{\theta}{\alpha}\Big)\right),
$$
and $\cos\Big(\frac{\theta}{\alpha}\Big) < 0$ since $\frac{\pi}{2} < \frac{\theta}{\alpha} < \pi$.
 Put 
\[
M_l(\alpha,\lambda)=\frac{l!}{2\alpha \pi (|\lambda|\sin \theta_0)^{l+1}}\int_{\gamma(\varepsilon,\theta)}|\exp{(\zeta^{\frac{1}{\alpha}})}|d\zeta.
\]
This together with \eqref{InE1} completes the proof of this part.

\noindent (ii) By Lemma~\ref{lemma1}, for any $t>0$ we have
\begin{equation}\label{IntEq1}
E_{\alpha,\alpha}(\lambda t^\alpha)=\dfrac{1}{2\alpha\pi i}\int_{\gamma(1,\theta)}\frac{\exp{(\zeta^{\frac{1}{\alpha}})}\zeta^{\frac{1-\alpha}{\alpha}}}{\zeta-\lambda t^\alpha}d\zeta. 
\end{equation}
Substituting the representation 
\[
\frac{1}{\zeta-\lambda t^\alpha}=-\frac{1}{\lambda t^\alpha}+\frac{\zeta}{\lambda t^\alpha (\zeta-\lambda t^ \alpha)}
\]
into \eqref{IntEq1} we get
\[
E_{\alpha,\alpha}(\lambda t^\alpha)=-\frac{1}{\lambda t^\alpha}\dfrac{1}{2\alpha\pi i}\int_{\gamma(1,\theta)}\exp{(\zeta^{\frac{1}{\alpha}})}\zeta^{\frac{1-\alpha}{\alpha}}d\zeta+\dfrac{1}{2\alpha\pi i}\frac{1}{\lambda t^\alpha}\int_{\gamma(1,\theta)}\frac{\exp{(\zeta^{\frac{1}{\alpha}})}\zeta^{\frac{1}{\alpha}}}{\zeta-\lambda t^\alpha}d\zeta.
\]
Note that
\[
\dfrac{1}{2\alpha\pi i}\int_{\gamma(1,\theta)}\exp{(\zeta^{\frac{1}{\alpha}})}\zeta^{\frac{1-\alpha}{\alpha}}d\zeta={\frac{1}{\Gamma(z)}}\Big|_{z=0}=0,
\]
see Podlubny~\cite[Formula (1.52), p.~16]{Podlubny}.
Hence
\[
E_{\alpha,\alpha}(\lambda t^\alpha)=\dfrac{1}{2\alpha\pi i}\frac{1}{\lambda t^\alpha}\int_{\gamma(1,\theta)}\frac{\exp{(\zeta^{\frac{1}{\alpha}})}\zeta^{\frac{1}{\alpha}}}{\zeta-\lambda t^\alpha}d\zeta,
\]
for all $t>0$. Using the Leibniz rule for the $l$th derivative of a product of two factors, we obtain
\begin{align*}
\frac{d^l}{d\lambda^l}E_{\alpha,\alpha}(\lambda t^\alpha)=\frac{1}{2\alpha \pi i t^\alpha}\sum_{k=0}^l C_l^k \frac{(-1)^{l-k}(l-k)!}{\lambda^{l-k+1}}\int_{\gamma(1,\theta)}
\frac{\exp{(\zeta^{\frac{1}{\alpha}})}\zeta^{\frac{1}{\alpha}}k!t^{\alpha k}}{(\zeta-
\lambda t^\alpha)^{k+1}}d\zeta,
\end{align*}
which again together with \eqref{InE} yields
$$
|\frac{d^l}{d\lambda^l}E_{\alpha,\alpha}(\lambda t^\alpha)|  \le \frac{1}{2\alpha\pi |\lambda|^{l+2}}\sum_{k=0}^l C_l^k \frac{(l-k)!k!}{(\sin \theta_0)^{k+1}}
\times \frac{1}{t^{2\alpha}}
\int_{\gamma(1,\theta)}|
\exp{(\zeta^{\frac{1}{\alpha}})}\zeta^{\frac{1}{\alpha}}|d\zeta,
$$
provided that $t >  t_0$. Noting again that the integral in the above formula converges.
Setting
$$
\hat{M}_l(\alpha,\lambda) := \frac{1}{2\alpha\pi |\lambda|^{l+2}}\sum_{k=0}^l C_l^k \frac{(l-k)!k!}{(\sin \theta_0)^{k+1}}\int_{\gamma(1,\theta)}|
\exp{(\zeta^{\frac{1}{\alpha}})}\zeta^{\frac{1}{\alpha}}|d\zeta
$$
we complete the proof.
\end{proof}
\begin{lemma}\label{thm.spec.estimate}
Let $A\in \R^{d\times d}$. Assume that the spectrum of $A$ satisfies the relation 
$$
\sigma(A)\subset\left\{\lambda\in \C\setminus\{0\}:|\arg(\lambda)|>\frac{\alpha \pi}{2}\right\}. 
$$
Then, the following statements hold:
\begin{itemize}
\item [(i)]  $\lim_{t\to \infty}\|E_\alpha(t^\alpha A)\|=0;$
\item [(ii)]  $\int_0^\infty \tau^{\alpha -1}\| E_{\alpha,\alpha}(\tau^\alpha A)\| \,d\tau <\infty.$
\end{itemize}
\end{lemma}
\begin{proof}
Denote by $a_1,\dots, a_r$ the eigenvalues of the matrix $A$.  From linear algebra we know that there exists a nonsingular matrix $T\in\C^{d \times d}$ transforming $A$ into the Jordan normal form $B$, i.e.
\[
B=T^{-1}A T=\hbox{diag}(B_1,\dots, B_s),
\]
where the Jordan block $B_i$, $1\le i\le s$, is of the following form
\[
    \left(
      \begin{array}{*5{c}}
       \lambda_i    &     1         &    0      &  \cdots       &  0        \\
                 0                       & \lambda_i     &    1      &     \cdots    &      0     \\
                 \vdots                      &      \cdots         &\ddots         &  \cdots   &  \vdots    \\
                 0                       &       0        &       \cdots    &  \lambda_i &          1 \\
                             0          &    0  &      \cdots     &    0     &\lambda_i \\
      \end{array}
    \right)_{d_i\times d_i}
\]
with $\lambda_i\in \{a_1,\ldots,a_r\}$ and $\sum_1^s d_i=d$, see Lancaster and Tismenetsky~\cite[\S6.5, Corollary 1, p.~237]{Lancaster}.
For each $i \in \{1,\dots,s\}$, we have
\begin{align*}
&E_{\alpha,\beta}(B_it^\alpha)\quad =
\quad \sum\limits_{k=0}^\infty\dfrac{(B_it^\alpha)^k}{\Gamma(\alpha k+\beta)}\quad =\quad \sum\limits_{k=0}^\infty\dfrac{t^{\alpha k}B_i^k}{\Gamma(\alpha k+\beta)}\\[8pt]
&=\sum\limits_{k=0}^\infty\dfrac{t^{\alpha k}}{\Gamma(k\alpha+\beta)}
\left(
\begin{array}{ccccc}
C_k^0 \lambda_i^k& C_k^1\lambda_i^{k-1}& C_k^2\lambda_i^{k-2} &\cdots & C_k^{d_i-1}\lambda_i^{k-d_i+1}\smallskip\\
 0 &C_k^0 \lambda_i^k& C_k^1\lambda_i^{k-1}&\cdots & C_k^{d_i-2}\lambda_i^{k-d_i+2}\smallskip \\
        \vdots                      &      \cdots         &\ddots         &  \cdots   &  \vdots  \smallskip  \\
 0& 0&\cdots & C_k^0 \lambda_i^k &C_k^1\lambda_i^{k-1}\smallskip\\
 0 & 0&  \cdots&0&C_k^0 \lambda_i^k
\end{array}
\right)\\[8pt]
&=\sum\limits_{k=0}^\infty\dfrac{t^{\alpha k}}{\Gamma(k\alpha+\beta)}
\left(
\begin{array}{ccccc}
 \lambda_i^k& \dfrac{1}{1!}\dfrac{\partial}{\partial \lambda_i}(\lambda_i^k)& 
\dfrac{1}{2!}\dfrac{\partial^2}{\partial \lambda_i^2}(\lambda_i^k) &
\cdots & \dfrac{1}{(d_i-1)!}\dfrac{\partial^{d_i-1}}{\partial \lambda_i^{d_i-1}}(\lambda_i^k)\smallskip\\
 0 & \lambda_i^k& \dfrac{1}{1!}\dfrac{\partial}{\partial \lambda_i}(\lambda_i^k)&\cdots & \dfrac{1}{(d_i-2)!}\dfrac{\partial^{d_i-2}}{\partial \lambda_i^{d_i-2}}(\lambda_i^k)\smallskip \\
        \vdots                      &      \cdots         &\ddots         &  \cdots   &  \vdots  \smallskip  \\
 0& 0&\cdots &  \lambda_i^k &\dfrac{1}{1!}\dfrac{\partial}{\partial \lambda_i}(\lambda_i^k)\smallskip\\
 0 & 0&  \cdots&0& \lambda_i^k
\end{array}
\right),
\end{align*}
here 
\begin{equation*}
C_k^j=
\begin{cases}
\dfrac{k!}{j!(k-j)!}, &\text{if } k\ge j,\medskip\\
0, & \text{otherwise},
\end{cases}
\end{equation*}
which are the binomial coefficients for $j\leq k$.
This implies 
\[
E_{\alpha,\beta}(B_it^\alpha)=\left(
\begin{array}{ccccc}
E_{\alpha,\beta}(\lambda_it^\alpha)&\dfrac{1}{1!}\dfrac
{\partial}{\partial \lambda_i}E_{\alpha,\beta}(\lambda_it^\alpha)&\cdots&
\dfrac{1}{(d_i-1)!}\dfrac{\partial^{d_i-1}}{\partial \lambda_i^{d_i-1}}E_{\alpha,\beta}(\lambda_1t^\alpha)\smallskip\\
0 &E_{\alpha,\beta}(\lambda_it^\alpha)&\cdots &
\dfrac{1}{(d_i-2)!}\dfrac{\partial^{d_i-2}}{\partial \lambda_i^{d_i-2}}E_{\alpha,\beta}(\lambda_1t^\alpha)\smallskip\\
 \vdots&\cdots&\ddots&\vdots\smallskip\\
0&0&\cdots&E_{\alpha,\beta}(\lambda_it^\alpha)
\end{array}
\right),
\]
cf. Lancaster and Tismenetsky~\cite[\S9.4, Theorem~4, p.~311]{Lancaster}.
Due to the fact that $E_{\alpha}(At^\alpha)=TE_{\alpha}(Bt^\alpha)T^{-1}$ and $E_{\alpha,\alpha}(At^\alpha)=TE_{\alpha,\alpha}(Bt^\alpha)T^{-1}$, to complete the proof of the theorem, it is suffices to prove
$$\lim_{t\to \infty}\|E_\alpha(t^\alpha B)\|=0$$ and
$$\int_0^\infty \tau^{\alpha -1}\| E_{\alpha,\alpha}(\tau^\alpha A)\| \,d\tau<\infty.$$
\noindent (i) By Theorem \ref{thm12}(i), for any $i\in \{1,\dots,s\}$ and $l\in \{0,\dots,d_i-1\}$, we have
\[
\lim_{t\to \infty}|\frac{d^l}{d\lambda_i^l} E_{\alpha}(\lambda_i t^{\alpha})|=0.
\]
Hence 
$$\lim_{t\to \infty}\|E_\alpha(t^\alpha B)\|=0.$$
\noindent (ii) For $t_0$ specified in Theorem~\ref{thm12}, since $\alpha>0$ we have 
\begin{equation}\label{finiteE}
\int_0^{t_0} \tau^{\alpha-1}\|E_{\alpha,\alpha}(\tau^\alpha B)\|\,d\tau <\infty.
\end{equation}
Therefore, to show that $$\int_0^\infty \tau^{\alpha -1}\| E_{\alpha,\alpha}(\tau^\alpha B)\| \,d\tau<\infty,$$
it is suffices to prove that 
$$\int_{t_0}^\infty \tau^{\alpha -1}\| E_{\alpha,\alpha}(\tau^\alpha B)\| \,d\tau<\infty.$$
By Theorem~\ref{thm12}(ii), for any $i\in \{1,\dots,s\}$ and $l\in \{0,\dots,d_i-1\}$, we have
\[
\int_{t_0}^\infty \tau^{\alpha-1}|\frac{d^l}{d \lambda_i^l}E_{\alpha,\alpha}(\lambda_i \tau^\alpha)|\,d\tau
\leq \int_{t_0}^\infty\frac{\hat{M}_l(\alpha,\lambda)}{\tau^{\alpha+1}} \,d\tau
=\frac{\hat{M}_l(\alpha,\lambda)}{\alpha t_0^{\alpha}}.
\]
Therefore, 
$$\int_{t_0}^\infty \tau^{\alpha -1}\| E_{\alpha,\alpha}(\tau^\alpha B)\| \,d\tau<\infty,$$
which together with \eqref{finiteE} implies
$$\int_0^\infty \tau^{\alpha -1}\| E_{\alpha,\alpha}(\tau^\alpha B)\| \,d\tau<\infty.$$
The proof is complete.
\end{proof}

\end{document}